\documentclass[11pt]{amsart}

\usepackage{amsmath,amssymb,amscd,amsfonts,verbatim}
\usepackage{paralist}
\usepackage[mathscr]{eucal}
\usepackage{color}

\usepackage[pdftex]{hyperref}
\hypersetup{citecolor=blue,linktocpage}

\newtheorem{thm}{Theorem}[section]
\newtheorem{lem}[thm]{Lemma}

\newtheorem{cor}[thm]{Corollary}

\newtheorem{claim}[thm]{Claim}
\newtheorem{assu-nota}[thm]{Assumption--Notation}

\theoremstyle{definition}

\newtheorem{rem}[thm]{Remark}

\newcommand{\inv}{^{-1}}

\newcommand{\into}{\hookrightarrow}

\newcommand{\Z}{\mathbb Z}

\newcommand{\pp}{\mathbb P}

\newcommand{\OO}{\mathcal O}

\DeclareMathOperator{\Pic}{Pic}

\DeclareMathOperator{\Id}{Id}

\newcommand{\bbP}{\mathbb{P}}

\newcommand{\bbZ}{\mathbb{Z}}

\newcommand{\epsi}{\varepsilon}

\numberwithin{equation}{section}

\title{Abelian varieties in Brill--Noether loci}
\author{Ciro Ciliberto}
\address{Dipartimento di Matematica, II Universit\`a di Roma,
Italy}
\email{cilibert@axp.mat.uniroma2.it}
\author{Margarida Mendes Lopes}
\address{Departamento de Matem\'atica, Instituto Superior T\'ecnico, Universidade T\'ecnica de Lisboa Av. Rovisco Pais,
1049-001 Lisboa, Portugal}
\email{ mmlopes@math.ist.utl.pt}

\author{Rita Pardini}
\address{Dipartimento di Matematica Universit\`a di Pisa\\
Largo B. Pontecorvo, 5,  56127 Pisa, Italy }
\email{pardini@dm.unipi.it}

\thanks{
The first and third authors are members of GNSAGA of INdAM, and this research was partially supported by the MIUR--PRIN 2010
``Geometria delle variet\`a algebriche''. The second author is a member of the Center for Mathematical
Analysis, Geometry and Dynamical Systems (IST/UTL).    This research was partially supported by FCT (Portugal) through program POCTI/FEDER and
Project PTDC/MAT/099275/2008.}

\begin{document}
\begin{abstract} 
In this paper, improving on results in \cite {ah, df}, we give the full classification of curves $C$ of genus $g$ such that a Brill--Noether locus $W^ s_d(C)$, strictly contained in the jacobian $J(C)$ of $C$, contains a variety $Z$ stable under translations by the elements of a positive dimensional abelian subvariety $A\subsetneq J(C)$ and such that $\dim(Z)=d-\dim(A)-2s$, i.e., the maximum possible  for such a $Z$.
\medskip

\end{abstract} 
\maketitle
\tableofcontents

\section{Introduction}

In  \cite {ah} the authors posed the  problem of studying, and possibly classifying, situations like this:
\begin{itemize}
\item [(*)] $C$ is a smooth, projective, complex curve of genus $g$,  $Z$ is an irreducible $r$--dimensional subvariety of a Brill--Noether locus $W^ s_d(C)\subsetneq J^ d(C)$, and $Z$ is stable under translations by the elements of an abelian subvariety $A\subsetneq J(C)$ of dimension $a>0$ (if so, we will say that $Z$ is $A$--\emph{stable}).
\end{itemize}

Actually in \cite {ah} the variety $Z$ \emph{is} the translate of a positive dimensional  proper abelian subvariety of $J(C)$, while the above slightly more general formulation  was given in \cite {df}. 

 The motivation for studying  (*) resides, among other things,  in a theorem of Faltings (see \cite {f}) to the effect that if $X$ is an abelian variety defined over a number field $\mathbb K$, and $Z\subsetneq X$ is a subvariety  not containing any translate of a positive dimensional abelian subvariety of $X$, then the number of rational points of $Z$ over $\mathbb K$ is finite. The idea in  \cite {ah} was to apply Faltings' theorem to the $d$--fold symmetric product $C(d)$ of a curve $C$ defined over a number field $\mathbb K$.
If $C$ has no positive dimensional linear series of degree $d$, then $C(d)$ is isomorphic to its Abel--Jacobi image $W_d(C)$ in $J^ d(C)$. Thus $C(d)$ has finitely many rational points over $\mathbb K$ if $W_d(C)$ does not contain any translate of a positive dimensional abelian subvariety of $J(C)$. The suggestion in \cite {ah} is that, if, by contrast, $W_d(C)$ contains the translate of a positive dimensional abelian subvariety of $J(C)$, then $C$ should be \emph{quite special}, e.g., it should admit a map to a curve of lower positive genus (curves of this kind clearly are in situation (*)). This idea was tested in \cite {ah}, where a number of partial results  were proven for low values of $d$.

The problem was taken up in \cite {df}, see also \cite{df-errata},  where, among other things, it is proven that if (*) holds, then $r+a+2s\leqslant d$, and, if in addition $d+r\leqslant g-1$, then $r+a+2s= d$ if and only if:
\begin{itemize}
\item [(a)]  there is a degree 2 morphism $\varphi\colon C\to C'$, with $C'$ a smooth curve of genus $a$, such that 
$A=\varphi^*(J(C'))$ and $Z=W_{d-2a-2s}(C)+\varphi^*(J^{a+s}(C'))$.
\end{itemize}

In \cite {df} there is also the following example with $(d,s)=(g-1,0)$:
\begin{itemize}
\item [(b)]  there is an (\'etale) degree 2 morphism $\varphi\colon C\to C'$, with $C'$ a smooth curve of genus $g'=r+1$, $A$ is  the  Prym variety of $\varphi$ and $Z\subset W_{g-1}(C)$ is the connected component of $\varphi_*\inv(K_{C'})$ consisting of  divisor classes $D$ with $h^0(\mathcal O_C(D))$ odd, where $\varphi_*\colon J^{g-1}(C)\to J^{g-1}(C')$ is the \emph{norm map}.  One has $Z\cong A$, hence $r=a$.\end{itemize}
One more family of examples is the following (see Corollary \ref {for:ten} below):
\begin{itemize}
\item [(c)]  $C$ is hyperelliptic, there is a degree 2 morphism $\varphi\colon C\to C'$ with $C'$ a smooth curve of genus $a$ such that $g>2a+1$,  $A=\varphi^*(J(C'))$, $0<s<g-1$ and $Z=\varphi^*(J^{a}(C'))+W_{d-2s-2a}(C)+ W^ s_{2s}(C)$ (notice  that $W^s_{2s}(C)$ is a point). 
\end{itemize}

The aforementioned result in \cite {df} goes exactly in the direction indicated in \cite {ah}. The unfortunate feature of it is the hypothesis $d+r\leqslant g-1$ which turns out to be \emph{quite strong}. To understand how strong it is, consider the case $(d,s)=(g-1,0)$, which is indeed the crucial one (see \cite [Proposition 3.3] {df} and \S \ref {ssec:GBN} below) and in which Debarre--Fahaloui's theorem is void. 

The aim of the present paper is to give the full classification of the cases in which (*) happens and $d=r+a+2s$. What we prove (see Theorem \ref{thm:main} and Corollary \ref {for:ten}) is that if (*) holds then, with no further assumption, either (a) or (b) or (c) occurs. 

The idea of the proof is not so different, in principle, from the one proposed in \cite {df} in  the restricted situation considered there. Indeed, one uses the $A$--stability of $Z$ and its maximal dimension  to produce linear series on $C$ which are not birational, in fact  composed with a degree 2 irrational involution. The main tool in  \cite {df}, inspired by \cite {ah}, is a Castelnuovo's type of analysis for  the growth of the dimension of certain linear series.

Our approach also consists  in producing a non birational linear series on $C$, but it is in a sense more direct. We consider  (*) with $(d,s)=(g-1,0)$ and $a+r=g-1$, i.e., the basic case (the others follow from this), in which $Z$ is contained in $W_{g-1}(C)$, which is a translate of the theta divisor $\Theta\subset J(C)$. This
 immediately produces, restricting to $Z$  the Gauss map of $\Theta$, a base point free sublinear series $L$ of dimension $r$ of the canonical series of $C$. It turns out that $Z$ is birational to an irreducible component of the variety $C(g-1,L)\subset C(g-1)$ consisting of all  effective divisors of $C$ of degree $g-1$ contained in some divisor of $L$. The $A$--stability of $Z$ implies that $C(g-1,L)$ has some other component besides the one birational to $Z$, and this forces $L$ to be non--birational. Once one knows this, a (rather subtle) analysis of the map determined by $L$ and of its image leads to the conclusion. 

As for the contents of the paper, \S \ref {sec:symm} is devoted to a few general facts about the varieties $C(k,L)\subset C(k)$ of effective  divisors of degree $k$ contained in a divisor of a linear series $L$ on a curve $C$. These varieties, as we said,  play a crucial role in our analysis. Section \ref {sec:absub} is devoted to the proof of our main result. 

One final word about our own interest in this problem, which is quite different from the motivation of \cite {ah, df}. It is in fact  related to the study of irregular surfaces $S$ of general type, where situation (*) presents itself in a rather natural way. For example, let $C\subset S$ be a smooth,  irreducible curve, and assume that $C$ corresponds to the general point of an  irreducible component $\mathcal C$ of the Hilbert scheme of curves on $S$ which dominates $\Pic^ 0(S)$.  Recall that there is also only one irreducible component $\mathcal K$ of the Hilbert scheme of curves algebraically equivalent to canonical curves on $S$ which dominates $\Pic^ 0(S)$ (this is called the \emph{main paracanonical system}). The curves of $\mathcal C$ cut out on $C$ divisors which are residual, with respect to $\vert K_C\vert$, of divisors cut out by curves in  $\mathcal K$. Consider now the one of the two systems  $\mathcal C$ and $\mathcal K$ whose curves cut on $C$ divisors of minimal degree $d=\min\{C^ 2, K_S\cdot C\}$, and denote by $s$ the dimension of the general fibre of this system over $\Pic^ 0(S)$. Thus the image of 
the  natural  map $\Pic^ 0(S)\to  J^d(C)$ defined by $\eta\mapsto \OO_C(C+\eta)$  (or by $\eta\mapsto \OO_C(K_S+\eta)$) is a $q$--dimensional abelian variety  contained in $W^s_d(C)$, which is what happens in (*). 
 Thus  understanding  (*) would provide us with the understanding of (most) curves on irregular surfaces. 

The results in this paper, even if restricted to the very special case of (*)  in which $Z$ has maximal dimension, turn out to be useful in surface theory. For example, if $S$ is a minimal irregular  surface of general type, then  $K^ 2_S\geqslant 2p_g$ (see \cite {debarre}). Using the results in this paper we are able to classify irregular minimal surfaces for which $K^ 2_S= 2p_g$. 
This classification  is given   in \cite{cmp}.

\section{Linear series on curves and related subvarieties of symmetric products}\label{sec:symm}

\subsection{Generalities} Let $C$ be  a smooth, projective, irreducible   curve of genus $g$. For an integer $k\geqslant 0$, we denote by $C(k)$ the $k$-th symmetric product of $C$ (by convention, $C(0)$ is a point). 
 
 Let $L$ be a base point free $g^r_d$ on $C$. We denote the corresponding morphism by $\phi_L\colon C\to \bar C \subset \bbP^ r$ and by  $C'$  the normalization of  $\bar C$. We let $f\colon C\to C'$ be the induced morphism and  $L'$ the linear series on $C'$ such that $L=f^ *(L')$. We set $\deg(f)=\nu\ge 1$, so that $d=\delta \nu$, with $\delta=\deg(\bar C)$. We say that $L$ is \emph{birational} if  $\phi_L\colon C\to \bar C$ is birational.

 Let $k\leqslant d$ be a positive integer. We consider the incidence correspondence
 $$\mathcal C{(k,L)} =\{(D,H)\in C(k)\times L \ | \ D\leqslant H\}$$
 with projections $p_i$ (with $1\leqslant i\leqslant 2$) to the first and second factor. Set 
 $C(k,L)=p_1(\mathcal C(k,L))$, which has a natural scheme structure (see \cite {c83} or \cite [p. 341]  {acgh}).  
  Note the isomorphism 
 \[\mathfrak s\colon  (D,H)\in \mathcal C{(k,L)}\to (H-D,H)\in \mathcal C{(d-k,L)}.\]

   \begin{lem}\label{lem:irreducible} 
    \begin{inparaenum}
   \item If $k\leqslant r$, then $C(k,L)=C(k)$.\\
   \item If $k\geqslant r$, then $C(k,L)$ has pure dimension $r$. If $V$ is an irreducible component of 
   $C(k,L)$, there is a unique component $\mathcal V$ of $\mathcal C(k,L)$ dominating $V$ via $p_1$ and $p_1$ induces a birational morphism 
   of $\mathcal V$ onto $V$. Finally $\mathcal V$ dominates $L$ via $p_2$.\\
   \item If $\min\{k,d-k\}\geqslant r$ the isomorphism $\mathfrak s$ induces a map $\mathfrak r\colon  C(k,L)\dasharrow C(d-k,L)$, which is componentwise birational.\\
      \item If $L$ is birational, then $C(k,L)$ is irreducible.
   \end{inparaenum}
 \end{lem} 
 \begin{proof}
Part (i) is clear. The dimensionality assertion in (ii) follows from \cite[\S~1]{c83} or  \cite [Lemma (3.2), p. 342]  {acgh}. The rest of (ii) follows  from these facts: the fibres of $p_1$ are isomorphic to linear subseries of $L$ and $p_2$ is finite, so no component of $\mathcal C(k,L)$ has dimension larger than $r$. Part (iii) follows from (ii). Part (iv) follows from the Uniform Position Theorem (see \cite [p. 112] {acgh}). \end{proof}

 Next we look at the case  $L$ non--birational and $k\geqslant r$. 
Consider the induced  finite morphism $f_*\colon C(k,L)\to C'(k)$. 
For each partition $\underline m=(m_1, \dots m_s)$ of $k$ with  $\delta \geqslant s\geqslant r$ and $1\leqslant  m_1\leqslant m_2\leqslant \dots \leqslant m_{s}\leqslant \nu$, we denote by $C( \underline m, L)$ the closure in $C(k, L)$ of the inverse image via $f_*$  of the set of  divisors of the form $m_1y_1+\dots +m_sy_s$, with $y_1+\dots+ y_s$ a reduced divisor 
in $C'(s,L')$.  To denote a partition $\underline m$ as above we may use the \emph{exponential notation} $\underline m=(1^{\mu_1},\ldots, \nu^ {\mu_\nu})$,
meaning that $i$ is repeated $\mu_i$ times, with $1\leqslant i\leqslant \nu$, and we  may omit $i^ {\mu_i}$ if $\mu_i=0$.  Note that $\sum_{i=1}^\nu  \mu_i=s\leqslant \delta$. 
Set $\mu_0:=\delta-s$.  Then $\underline m^ c=(1^{\mu_{\nu-1}},\ldots, (\nu-1)^ {\mu_1}, \nu ^ {\mu_0})$ is a partition of $d-k$ which we call the \emph{complementary partition} of $\underline m$. 

\begin{lem}\label{lem:reducible} In the above set up each irreducible component of  $C(\underline m, L)$ has dimension $r$, hence it is an irreducible component of  $C(k,L)$ and all irreducible components of $C(k,L)$ are of this type.  \end{lem} 
 \begin{proof} Let $V$ be an irreducible component of $C(k,L)$: by  Lemma \ref{lem:irreducible}  there is a unique   irreducible component $\mathcal V$ of $\mathcal C(k,L)$ dominating it and $p_1|_{\mathcal V}\colon \mathcal V\to V$ is birational.
  Let $D\in V$ be a general point and let $(D,H)$ be the unique point of $\mathcal V$ mapping to $D$ via $p_1$, so that $H$ is a general divisor in $L$  (cf. Lemma \ref {lem:irreducible}, (ii)). Hence
 $H$ consists of $\delta$ distinct fibres $F_1,\ldots, F_\delta$ of $f$, each being a reduced divisor of degree $\nu$ on  $C$. Then $D$ consists of $m_i$ points in $F_i$, for $1\leqslant i\leqslant s\leqslant \delta$, where we may assume $1\leqslant  m_1\leqslant m_2\leqslant \dots \leqslant m_{s}\leqslant \nu$. Moreover, since $\dim(V)=\dim(\mathcal V)=r$ and $p_2$ is finite, one has $s\geqslant r$.  Hence $D\in C(\underline m,L)$, with $\underline m=(m_1, \dots m_s)$, i.e.  $V\subseteq C(\underline m,L)\subseteq C(k,L)$ hence $V$ is a component of $C(\underline m,L)$.
  
The above considerations and Lemma  \ref {lem:irreducible}, (ii), applied to $C'(s,L')$, imply that the image of $C(\underline m,L)$ in $C'(k)$ has dimension $r$, so  each component of $C(\underline m,L)$ has  dimension $r$. \end{proof}

\subsection{Abel--Jacobi images}

 We assume from now on  that $C$ has  genus $g>0$.
For an integer $k$, we denote  by $J^k(C)\subset \Pic(C)$ the set of linear equivalence classes of divisors of degree $k$ on $C$.  So $J(C):=J^ 0(C)$ is the Jacobian of $C$, which is a principally polarised abelian variety whose theta divisor class we denote by $\Theta_C$, or simply by $\Theta$. 

The abelian variety $J(C)$  acts via translation on $J^ k(C)$ for all $k$. 
If $X\subseteq J^ k(C)$ and $Y\subseteq J(C)$, we say that $X$ is \emph{$Y$--stable}, if for all $x\in X$ and for all $y\in Y$, one has $x+y\in X$. 

For all integers $k$, fixing the class of a divisor of degree $k$  determines  an  isomorphism $J^ k(C)\cong J(C)$. Given a subvariety $V$ of  $J^ k(C)$, one says that it \emph{generates} $J^ k(C)$ if the image of $V$ via  one of the above isomorphisms generates  $J(C)$ as an abelian variety. This  definition     does not depend on the choice of the isomorphism  $J^ k(C)\cong J(C)$. 

For every $k\geqslant 1$, we denote by $j_k \colon C(k)  \to J^ k(C)\cong J(C)$ (or simply by $j$) the Abel--Jacobi map. We denote  by $W^ s_k(C)$ the subscheme of $J^ k(C)$ corresponding to classes of divisors $D$ such that $h^ 0(\mathcal O_C(D))\geqslant s+1$ (these are the so--called {\em Brill--Noether loci}). One sets $W_k(C):=W^ 0_k(C)={\rm Im}(j_k)$ and $W_{g-1}(C)$ maps to  a theta divisor of $J(C)$, so we may abuse notation and write $W_{g-1}(C)=\Theta_C$. 

We denote by $\Gamma_C(k, L)$ [resp. $\Gamma_C(\underline m, L)$]
the image in $W_k(C)$ of $C(k,L)$ [resp. $C(\underline m,L)$] (we may drop the subscript $C$ if there is no danger of confusion). The expected dimension of $\Gamma(k,L)$ is $\min\{r,g\}$ 
(by dimension of a scheme we mean the maximum of the dimensions of its components). 

We set $\rho_C(k,L):= \dim(\Gamma_C(k,L))$ (simply denoted by $\rho(k,L)$ or by $\rho$ if no confusion arises). By Lemma \ref {lem:irreducible}, (i), one has:
\begin{equation}\label{eq:aiuto}
{\rm if} \,\,\, k\leqslant r \,\, \, {\rm then}\,\,\, \Gamma(k,L)=W_k(C),\,\,\, {\rm hence}\,\,\, \rho=\min\{k,g\}
\end{equation}
So we will  consider next the case $k>r$, in which  $\rho\leqslant \dim(C(k,L))=r$, by Lemma  \ref {lem:irreducible}, (ii). 
Then the class $c (k,L)$ of $C(k, L)$ in the Chow ring of $C(k)$ is computed in \cite[ Lemma VIII.3.2] {acgh}. If  $x$ is the class of $C(k-1)\subset C(k)$ and $\theta:=j^ *(\Theta)$, one has
\begin{equation}\label{eq:class}
c(k,L)=\sum_{s=0}^{k-r} {{d-g-r}\choose{s}}\frac {x^s\theta^{k-r-s}}{(k-r-s)!} \ .
\end{equation}

\begin{lem}\label{class} Assume $k>r$  and $d-g-r \geqslant 0$. Then:\\
\begin{inparaenum}[(i)]
\item if $k-g\leqslant \min\{k-r, d-g-r\}$ one has $\rho=r\leqslant g$;\\
\item if $k-g\geqslant \min\{k-r, d-g-r\}=k-r$, one has $\rho=g\leqslant r$;\\
\item if $k-g\geqslant \min\{k-r, d-g-r\}=d-g-r$, one has $\rho=d-k\leqslant \min\{r,g\}$.
\end{inparaenum}
\end{lem}

\begin{proof} Note that $x^ s$ is the class of $C(k-s)\subset C(k)$, for $1\leqslant s\leqslant k$. Applying  the projection formula (cf. \cite{fulton}, Example 8.1.7)  to \eqref {eq:class}, we find the class $\gamma(k,L)$ of $\Gamma(k,L)$ 
\begin{equation}\label{eq:class2}
\gamma(k,L)=\sum_{s=0}^{k-r} {{d-g-r}\choose{s}}\frac {w_{k-s}\Theta^{k-r-s}}{(k-r-s)!} \ 
\end{equation}
where $w_i$ is the class of $W_i(C)$ for any $i\geqslant 0$. By Poincar\'e's formula (cf. \cite[\S 11.2]{birk}) one has
\[
w_{k-s}\Theta^{k-r-s}=
\begin{cases}
\Theta^{k-r-s}, & \text{if $k-g\geqslant s\geqslant 0$,} \\
\frac {\Theta^{g-r}}{(g-k+s)!}, & \text{if $\max\{0,k-g\}\leqslant s\leqslant  \min\{k-r, d-g-r\}$},
\end{cases}
\]
whence the assertion follows. \end{proof}

\begin{lem}\label{lem:class2}
\begin{inparaenum}[(i)]
\item If $\rho=g$ then $g\leqslant \min\{k,r\}$;\\
\item  if $r\geqslant k\geqslant g$, then $\rho=g$;\\
\item  if $k> r\geqslant g$ and $d\geqslant k+g$ then $\rho=g$;\\
\item  $\rho=0$ if and only if $k=d$.
\end{inparaenum}
\end{lem}

\begin{proof} Parts (i) and (ii) follow from Lemma \ref {lem:irreducible}. 

(iii) In \eqref {eq:class2} one has the summand corresponding to the index $s=k-r>0$, which is $\Theta^0$ with the positive coefficient ${{d-g-r}\choose{k-r}}$, and no other summand in  \eqref {eq:class2} cancels it. 

(iv) If $k=d$ then $C(k,L)=L$ and clearly $\rho=0$. Conversely, if $\rho=0$ then in \eqref {eq:class2} the term $\Theta^ g$ has to appear with non--zero coefficient and no other term $\Theta^ i$  with $0\leqslant i<g$  appears with non--zero coefficient.  By looking at the proof of Lemma \ref {class}, we see that  the summand $\Theta^ g$ appears in \eqref  {eq:class2} only if $0\leqslant s=k-r-g$. Then $d\geqslant k\geqslant r+g$. So we may apply Lemma \ref {class}, and conclude that $\rho=0$ occurs only in case (iii), if $k=d$. \end{proof}

\begin{lem}\label{lem:genera} 
Let $A\subseteq J(C)$ be an  abelian subvariety of dimension $a$ and let $p\colon J^k(C)\to J':=J(C)/A$ be the map  obtained by composing an isomorphism $J^k(C)\cong J(C)$ with the quotient map $J(C)\to J'$. Then 
$$\dim (p(\Gamma(k,L)))=\min\{g-a, \rho\}. $$

\end{lem}
\begin{proof} 
If $\rho=g$ the statement is obvious, hence we assume $\rho<g$. 

Consider first the case $k>r$.  Assume  by 
 contradiction that $\dim (p(\Gamma(k,L))<\min\{\rho, g-a\}$.
Let $\xi$ be the class of the pull back to $J(C)$ of an ample line bundle of $J'$. We have $\bar \gamma(k,L)\xi^{\rho}=0$, where $\bar\gamma(k,L)$ is the $\rho$--dimensional part of $\gamma(k,L)$. 
By \eqref {eq:class2} one has 
$\bar \gamma(k,L)= \alpha \Theta^ {g-\rho}$, 
where $\alpha\in \mathbb Q$ is positive because $\Gamma(k, L)$ is an effective non--zero cycle of dimension $\rho$. 
Hence $\bar \gamma(k,L)\xi^{\rho}=\alpha  \Theta^ {g-\rho}\xi^{\rho}>0$, because $\Theta$ is ample. Thus we have a contradiction.

If $k\leqslant r$, then $\Gamma(k,L)=W_k(C)$, $\rho=k$ and  $\gamma(k,L)$  is again a rational multiple of $\Theta^ {g-k}$ (by Poincar\'e's formula), so the proof proceeds as above. \end{proof}

\begin{cor} \label {cor:gen} If $A\subseteq J(C)$ is an  abelian subvariety of dimension $a>0$ and $\Gamma(k, L)$ is $A$--stable, then the restriction of $p$ to $\Gamma(k,L)$ is surjective onto $J'=J(C)/A$, hence $\Gamma(k, L)=J^ k(C)=W_k(C)$, i.e., $\rho=g$.
\end{cor}

\begin{cor} \label {cor:gen2} If $d-1 \geqslant k\geqslant 1$, then $\Gamma(k,L)$ generates $J^ k(C)$.
\end{cor}
\begin{proof}  If $k\leqslant r$ then $\Gamma(k,L)=W_k(C)$ and the assertion is clear. Assume $k>r$.
By Lemma \ref {lem:genera}, $\Gamma(k,L)$ generates $J^ k(C)$ as soon as $\rho>0$, which is the case by Lemma \ref {lem:class2}, (iv). \end{proof}

\subsection{A useful lemma}

Let  $L$ be  a base point free $g^ 1_d$, let  $\phi_L\colon  C\to \bbP^ 1$ be the corresponding map and denote by $G_L$ the Galois group of $\phi_L$. 
\begin{lem} \label{lem:tetra}
If  $L$ is a base point free  $g^1_4$, then one of the following occurs:\\
\begin{inparaenum}
\item[(a)]  $C(2,L)$ is irreducible.\\
\item[(b)]  $C(2,L)$ has two components.  This  occurs if and only if  $G_L\cong \bbZ_2, \,\, \bbZ_4$.\\
\item[(c)]  $C(2,L)$ has 3 components. This occurs if and only if  $G_L\cong \Z_2^2$. 
\end{inparaenum}
\end{lem}
\begin{proof}
We argue as in the proof  of \cite[Lemma 12.7.1]{birk}. Let $\Delta\subset \pp^1$ be the set of critical values of $\phi_L\colon  C\to \pp^ 1$,  let $\rho\colon \pi_1(\pp^1\setminus \Delta)\to \mathfrak S_4$ be the monodromy representation and let $\Sigma:={\rm Im}(\rho)$. The group $\Sigma$ acts transitively on $\mathbb I_4:=\{1,2,3,4\}$ (identified with the general divisor $x_1+\ldots+x_4\in L$) 
and its order $s$  is divisible by 4. The irreducible components of  $C(2,L)$ are in 1--to--1 correspondence with the $\Sigma$--orbits of the order two subsets of $\mathbb I_4$.
If $\Sigma$ contains an element of order 3, then $s$ is divisible by 12.  Hence $\mathfrak A_4\subseteq \Sigma$, thus the action is transitive, $C(2,L)$ is irreducible and  (a)  holds.

Assume  $\Sigma$ is a $2$--group. If $s=8$ then $\Sigma$ is a $2$--Sylow subgroup of $\mathfrak S_4$, hence $\Sigma$ is  the dihedral group $D_4$. Then  $C$ is obtained from a $\Sigma$--cover $C_0\to \pp^1$ by moding out 
 by a reflection $\sigma\in \Sigma$. In the $\Sigma$-action  on $\mathbb I_4$ we may assume that an  element of order 4 acts by sending 
 $i$ to ${i+1}$  modulo 4, for $1\leqslant i\leqslant 4$.  So the order 2 element in the center of  $\Sigma$ induces the involution $\iota$ of $C$ that maps $i$ to ${i+2}$ modulo 4, for $1\leqslant i\leqslant 4$, and $\iota$ generates $G_L\cong \bbZ_2$.  There are 
two orbits for the $\Sigma$--action on the set of order two subsets of $\mathbb I_4$: one of order 2 given by $\{\{1,3\}, \{2,4\}\}$, the other of order 4 given by $\{ \{i,i+1\}, \,\, \text {for}\,\, 1\leqslant i\leqslant 4\}$ (here $i$ is taken modulo $4$). These orbits  respectively correspond to  two components 
$E_1, E_2$ of $C(2,L)$ and we are in case (b). 

Assume  $s=4$. Then $\phi_L$ is Galois with  $G_L=\Sigma$. If $G_L\cong \mathbb Z_4$, then the $\Sigma$--orbits on the set of order two subsets of $\mathbb I_4$ are as in the previous case. 
If $G_L\cong \Z_2^2$, then 
one has $G_L=\{\Id, (12)(34), (13)(24),(14)(32)\}$. There are then three orbits, corresponding to three components $E_1,E_2,E_3$.\end{proof}

One can be more precise about the components $E_i$ of $C(2,L)$ in Lemma \ref {lem:tetra}, whose geometric genera we denote by $g_i$ (with $1\leqslant i\leqslant 2+\epsilon$, and $\epsilon =0$ in case (b), $\epsilon =1$ in case (c)).

\begin{lem}\label {lem:tetra2} Same setting and notation as in Lemma \ref {lem:tetra} and its proof.
Then:\\
\begin{inparaenum}[(i)]
\item each component of $C(2,L)$ maps birationally to its image in $J^ 2(C)$ unless $C$ is hyperelliptic and $L$ is  composed with the hyperelliptic involution $\mathcal L$: in this case one of the components of $\Gamma(2,L)$ is $\mathcal L=C(2,\mathcal L)$, which is contracted to a point in $J^ 2(C)$;\\
\item in case  (b) one has  $E_1\cong C/\iota$ (where $\iota$ is the non--trivial involution in $G_L$), which is hyperelliptic, and the abelian subvariety of $J^2(C)\cong J(C)$ generated by $j(E_1)$ is  the $\iota$--invariant part of $J(C)$.  Moreover  $2g_2\geqslant g$;\\
\item In case (c) one has $E_i:=C/\iota_i$ where $\iota_i$ are the three nonzero elements of $G_L$, for $1\leqslant i \leqslant 3$.
\end{inparaenum}
\end{lem}

\begin{proof} We prove the only non--trivial assertion, i.e., $2g_2\geqslant g$  in part (ii).

First assume $G_L=\Z_4=\langle \rho\rangle$. 
Consider in $C(2)$ the curves $E_1=\{P+\rho^2(P)\ |\ P\in C\}$ and  $E_2=\{P+\rho (P)\ |\ P\in C\}$.   One has  $C(2,L)=E_1\cup E_2$ and $E_1\cong C/\rho^2$. 
The curve $E_2$ is the image in $C(2)$ of the graph of $\rho$, so  $g_2=g$.

Suppose now $G_L=\Z_2$ and $\Sigma=D_4$.  Recall that $C$ is obtained from a $D_4$--Galois cover $f\colon  C_0\to \pp^1$ by moding out 
 by a reflection $\sigma\in D_4$ (see the proof of Lemma \ref {lem:tetra}). Denote by $g_0$ the genus of $C_0$. Let $\rho\in D_4$ be an element  of order 4, so that $D_4=\langle \sigma,\rho\rangle$. Let $n$ be the number of points of $C_0$ fixed by $\sigma$, $n'$ the number of points fixed by $\sigma \rho$,  $m$ the number of points fixed by $\rho$ and $m+\epsi$ the number of points fixed by $\rho^2$.  The Hurwitz formula, applied to $C_0\to \pp^ 1$ and to $C_0\to C=C_0/\sigma$, gives
\begin{equation}\label{eq:gC}
g_0=\frac 32m+n+n'+\frac \epsi 2-7, \,\,\,\, {\rm and} \,\,\,\,  g= \frac {n'} 2+\frac n 4+\frac \epsi 4-\frac 34 m-3.
\end{equation}

\begin{claim}\label{lem:rs} 
\begin{inparaenum}
\item $m$, $n$, $n'$ and $\epsi$  are even.\\
\item $n+m\equiv n'+m\equiv n+n'\equiv 0 \mod 4$, and at most one among $m,n,n'$ can be 0.
\end{inparaenum}
\end{claim}
\begin{proof} [Proof of the Claim]  (i) The numbers $n, n'$ and $m+\epsi$ are even because $\sigma, \sigma \rho$ and $\sigma \rho^ 2$ are involutions. If $P\in C_0$ is fixed by $\rho$, then $\sigma(P)$ is also fixed by $\rho$. Since  the stabilizer of any point is cyclic, then $\sigma(P)\neq P$. This implies that  $m$ is even.

(ii) Consider the $\mathbb Z_2^ 2$--cover  $D:=C_0/\rho^2 \to\pp^1$. The  cardinalities  of the images in $\pp^1$ of the fixed  loci of the three involutions are  $n/2$, $n'/2$ and $m/2$. 
Indeed, denote by $\gamma_1$ [resp. by $\gamma_2$]  the image  of $\sigma$ [resp.  of $\rho$]  in $D_4/\rho^2\cong \Z_2^2$. Let $Q\in \pp^1$ be a branch point whose preimage in $D$  is fixed   by   $\gamma_1$. Then the preimage of $Q$ in $C_0$  consists of 4 points,  two of which fixed by $\sigma$ and   two  by $\sigma\rho^2$, so the number of such points $Q$ is $(2n)/4=n/2$.  Similarly, the image  in $\pp^1$ of the set of points of $D$  fixed by $\gamma_2$ has cardinality $m/2$ and the image of the set of points fixed by $\gamma_1\gamma_2$ has cardinality $n'/2$.
Hurwitz formula for $D_1:=D/\gamma_1\to \pp^1$ gives 
 \[ 2g(D_1)-2=\frac m2+\frac {n'}2-4,
 \]
hence $m+n'>0$ is divisible by 4. Similarly  $n+n'$ and $m+n$ are positive and  divisible by 4.
\end{proof}

We compute now the ramification of $f\colon  C_0\to \pp^ 1$ and of  $L$. Write the $D_4$--orbit of $P\in C_0$ general as
\begin{equation}\label{eq:fibra}
\begin{array}{llll}
P& \rho (P)&\rho^2(P)&\rho^3(P)\\
\sigma (P)& \sigma \rho (P)&\sigma \rho^2(P)&\sigma \rho^3(P).
\end{array}
\end{equation}
Denote by $Q_1,\dots, Q_4$ the images in $C$ of  the points in the first row (or, what is the same, in the second row) of \eqref {eq:fibra}. The singular fibers of $f$ occur when $P$ has non trivial stabilizer, i.e., when:\\
\begin{inparaenum} [$\bullet$]
\item  $P\in C_0$ is  fixed by $\rho$. The fiber of $f$ is  $4(P+\sigma(P))$ and the corresponding divisor of 
$L$  is  $4Q_1$. There are $m/2$ divisors of $L$ of this type;\\
\item  $P$ is fixed by $\rho^2$ but not by $\rho$. The fiber of $f$ is  $2(P+\sigma(P)+\rho(P)+\sigma\rho(P))$.
Then $Q_1=Q_3$,  $Q_2=Q_4$, and the corresponding divisor of $L$ is $2(Q_1+Q_2)$. There are $\epsi/4$ such divisors;\\
\item  $P$ is fixed by $\sigma$.   The fibre of $f$ is $2(P+\rho(P)+\rho^ 2(P)+\rho^ 3(P))$, and 
$Q_2=Q_4$, while $Q_1,  Q_2$ and $Q_3$  are distinct, so the corresponding divisor of $L$ is $Q_1+2Q_2+Q_3$ and there are $n/2$ such divisors; \\
\item $P$ is fixed by $\sigma\rho^2$. This is the same as the previous case;\\
\item $P$ is fixed by $\sigma \rho$.  The fibre of $f$ is again $2(P+\rho(P)+\rho^ 2(P)+\rho^ 3(P))$, and 
\ $Q_1=Q_2$, $Q_3=Q_4$, so the corresponding divisor of $L$ is $2Q_1+2Q_3$ and there are $n'/2$ such fibers; \\
\item $P$ is fixed by $\sigma\rho^3$. This  is the same as the previous case.
\end{inparaenum}

Denote by $\iota$ the involution of $C$ induced by $\rho^2$. Then  $C(2,L)$ is the union of $E_1=\{P+\iota (P)\ |\ P\in C\}=C/\iota$ and of the irreducible curve $E_2$. Keeping the above notation, if $Q_1+\cdots+ Q_4$ is the general divisor of $L$, then $E_1$ is described by the divisors $Q_1+Q_3, Q_2+Q_4$ and $E_3$ by  $Q_1+Q_2, Q_1+Q_4, Q_2+Q_3, Q_3+Q_4$.  

To compute $g_2$,
define a map $\phi\colon C_0\to E_2$  by sending $P\in C_0$ to the image via $C_0\to C$ of the divisor 
$P+\rho(P)$, i.e.,  $Q_1+Q_2\in E_2$.

\begin{claim} \label{lem:gE_2}
\begin{inparaenum}
\item $\deg(\phi)=2$ and  $E_2$ is birational to $C_0/\sigma\rho$;\\
\item  $g_2=\frac 34 m+\frac n 2+\frac {n'}4+\frac \epsi 4-3$.
\end{inparaenum}
\end{claim}
\begin{proof} [Proof of the Claim]
(i)  Let $Q+Q'\in E_2$ be a general point and let  $P, \sigma (P)$ [resp. $P', \sigma (P')$] be the preimages of $Q$ [resp. of $Q'$] on $C_0$. Since $Q+Q'$ is of the form 
$Q_1+Q_2, Q_1+Q_4, Q_2+Q_3$, or  $Q_3+Q_4$,  
we may assume that $P'=\rho (P)$, so that $\psi^ {-1}(Q+Q')=\{P, \sigma\rho(P)\}$.

Part (ii) follows by applying Hurwitz formula.
\end{proof}

Finally, suppose by contradiction  that $2g_2\leqslant g-1$. 
Then, by (ii) of Claim \ref{lem:gE_2} and by \eqref{eq:gC}, we would have $3(m+n)+\epsi\le 8$, hence $m+n\le 2$, contradicting (ii) of Claim \ref{lem:rs}.\end{proof}

\section{Abelian subvarieties of Brill--Noether loci}\label{sec:absub}

As in  \cite{ah,df}, we consider $Z\subseteq W^s_d(C)\subsetneq J^ d(C)$ an irreducible  $A$--stable variety of dimension $r$,  with $A\subsetneq J(C)$ an abelian subvariety of dimension $a>0$. Note that $r\geqslant a$, with equality if and only if $Z\cong A$. Moreover, since $W^s_d(C)\subsetneq J^ d(C)$,  the general linear series $L\in Z$ is special,  thus $s>d-g$ and $d\geqslant 2s$ by Clifford's theorem. From \cite[Proposition 3.3]{df},  we have
\begin{equation}\label{eq:DF}
r+a+2s\leqslant d.
\end{equation}

 In this section, we  classify the cases in which equality holds in \eqref {eq:DF}, thus improving the partial results in \cite {df} on this subject.  Note that if equality holds in  \eqref {eq:DF}, then  $Z\not \subseteq  W^{s+1}_d(C)$ and  $A$ is a maximal abelian subvariety of $J(C)$ such that $Z$ is $A$--stable.

\subsection{The Theta divisor case}  As in \cite {df}, we first consider the case $(d,s)=(g-1,0)$,
i.e.,  $Z\subset W_{g-1}(C)=\Theta$ is an irreducible $A$--stable variety of dimension $r=g-1-a$. 
Then $Z\not\subseteq  W^ 1_{g-1}(C)={\rm Sing}(\Theta)$, i.e., $Z$ has a non--empty intersection with 
$\Theta_{\rm sm}:=\Theta - {\rm Sing}(\Theta)$.

\begin{thm}\label{thm:main}
Let $C$ be a curve of genus $g$. Let $A\subsetneq J(C)$ be an abelian variety of dimension $a>0$ and $Z\subset \Theta$ an irreducible,  $A$--stable variety of dimension $r=g-1-a$.
Then there is a degree 2 morphism $\varphi\colon C\to C'$, with $C'$ smooth of genus $g'$, such that one of the following occurs:\\
\begin{inparaenum}
\item[(a)] $g'=a$, $A=\varphi^*(J(C'))$ and $Z=W_{g-1-2a}(C)+\varphi^*(J^a(C'))$;\\
\item[(b)] $g'=r+1$, $\varphi$ is \'etale, $A$ is  the  Prym variety of $\varphi$ and $Z\subset W_{g-1}(C)$ is the connected component of $\varphi_*\inv(K_{C'})$ consisting of  divisor classes $D$ with $h^0(\mathcal O_C(D))$ odd, where $\varphi_*\colon J^{g-1}(C)\to J^{g-1}(C')$ is the {norm map}. 
\end{inparaenum}
In particular, $Z\cong A$ is an abelian variety if and only if either we are in case (a) and $g=2a+1$, or  in case (b). 
\end{thm}

\begin{rem} Cases (a) and (b) of Theorem \ref{thm:main} are not mutually exclusive. Indeed,  if the curve $C$ in case (b) is hyperelliptic, then the Abel-Prym map $C\to A$  induces a $2$-to-$1$ map $\psi\colon C\to D$, where $D$ is  a smooth curve  embedded into    $A\cong J(D)$ by the Abel-Jacobi map. One can  check that $A=\psi^*(J^r(D))$, namely  this is also an instance  of   case (a) of Theorem \ref{thm:main}.
\end{rem}

The proof of Theorem \ref{thm:main}  requires various preliminary lemmas.  First, recall that the tangent space to $J(C)$ at $0$ can be identified with $H^ 1(\mathcal O_C)\cong H^0(K_C)^*$. 
Denote by $T\subseteq H^0(K_C)^*$ the tangent space to $A$ at $0$ and  by $L$ the linear series $\pp(T^{\perp})\subseteq |K_C|$.  One has $\dim(L)=g-1-a=r$.

Since $Z\cap \Theta_{\rm sm}\not= \emptyset$,   the Gauss map of $\Theta$ restricts to a  rational map $\gamma \colon Z \dasharrow  \pp:=\vert K_C\vert$.

\begin{lem}\label{lem:im} One has $\gamma(Z)=L$.
\end{lem}

 \begin{proof}  Since $Z$ is $A$--stable, one has $\gamma(Z)\subseteq L$.  A point of $\Theta_{\rm sm}$ can be identified with a divisor $D$ of degree $g-1$ such that $h^ 0(\mathcal O_C(D))=h^ 0(\mathcal O_C(K_C-D))=1$.
 The Gauss map sends $D\in \Theta_{\rm sm}$ to the unique divisor of $\vert K_C\vert$ containing $D$.  Then $\gamma^ {-1}(\gamma(D))$ is finite. Hence $\dim(\gamma(Z))=\dim(Z)=r=\dim (L)$. The assertion follows. \end{proof}

 As in  \cite {df}, formula \eqref {eq:DF} for $(d,s)=(g-1,0)$ follows from the argument in the proof of Lemma \ref {lem:im}, and  the general case follows from this (see \S \ref {ssec:GBN} below).

Note
the birational involution $\sigma\colon  \Theta\dasharrow \Theta$,  defined on $\Theta_{\rm sm}$, sending a divisor $D\in \Theta_{\rm sm}$ to the unique effective divisor $D'\in \vert K_C-D\vert$. Then $\sigma$ restricts on  $Z$ to a birational map (still denoted by $\sigma$)  onto its image $Z'$, which is also $A$-stable. 

\begin{lem}\label{lem:free}
 The linear series $L$ is base point free.
\end{lem} 
\begin{proof}
Suppose  $P\in C$ is a base point of $L$. 
If every $D\in Z$ contains $P$, then the map $D\to D-P$ defines an injection  $Z\into W_{g-2}(C)$, contradicting   \eqref {eq:DF}.
If  the general $D\in Z$ does not contain $P$, then $\sigma(D)$ contains $P$ for  $D\in Z$  general, and we can apply the previous  argument to  $Z'$.
\end{proof}

\begin{lem}\label{lem:nobir} One has:\\
\begin{inparaenum} 
\item $Z$ is a component of $\Gamma(g-1, L)$.\\
\item $Z\subsetneq \Gamma(g-1, L)$.\\
\item  The linear series $L$ is not birational. 
\end{inparaenum}
\end{lem}
\begin{proof} By Lemma \ref {lem:im}, we have $Z\subseteq\Gamma(g-1, L)$ and the dimension of $\Gamma(g-1, L)$ is equal to $r$ (by Lemma \ref {class}).  So (i) holds. If $Z= \Gamma(g-1, L)$, then $\Gamma(g-1, L)$ is $A$-stable. By 
 Corollary \ref  {cor:gen}, we  have $r=g$, a contradiction. This proves (ii). Then (iii) holds by Lemma \ref{lem:irreducible}, (iv). 
 \end{proof}
 
Recall the notation $\phi_L\colon C\to \bar C\subset \pp^r$,  $C'$ for the normalization of  $\bar C$,  $f\colon C\to C'$ the induced morphism and  $\nu=\deg(f)>1$. For any integer $h$ one has the morphism $f^ *\colon J^ h(C')\to J^ {h\nu }(C)$. Let  $g'$ be the genus of $C'$ and $L'$ the (birational) linear series on $C'$ of dimension $r$ such that $L=f^*(L')$, whose degree is $\delta=\frac {2g-2}\nu$.

 \begin{lem} \label{lem:key}
 Assume $\nu \leqslant 3$ and let  $\underline m=(1^ {\mu_1}, \dots, \nu^ {\mu_\nu})$ be the partition of $g-1$ such that $Z\subseteq \Gamma(\underline m, L)$. Set $\mu_1=\mu$, $\mu_2=\mu'$. Then  $\nu=2$ and there are the following possibilities:\\
 \begin{inparaenum}
 \item $\mu=g-1$ (hence $\mu'=0$);\\
 \item $\mu=g-1-2a$, $\mu'=a$, $A=f^*(J(C'))$ (hence $g'=a$) and $Z=W_{\mu}(C)+f^*(J^a(C'))$.
 \end{inparaenum}
\end{lem}

 \begin{proof} 
 We first consider the case $\nu=2$.

 If $\mu=0$, then $g-1=2\mu'$  and  $Z=f^*(\Gamma_{C'}(\mu',L'))$, since $\Gamma_{C'}(\mu',L')$ is irreducible by Lemma \ref{lem:irreducible}.    The general point of $Z$ corresponds to a linearly isolated divisor, hence $\mu'\leqslant g'$.
 Since $f^ *\colon J^ {\mu'}(C')\to J^ {g-1}(C)$ is finite, one has $r=\rho_{C'}(\mu',L')\leqslant \mu'\leqslant g'$. Since $Z$ is $A$--stable, we have an isogeny $A\to \bar A\subseteq J(C')$ (so that $a\leqslant g'$) and $\Gamma_{C'}(\mu',L')$ is $\bar A$--stable.  
 Hence by  Corollary \ref {cor:gen} one has $\Gamma_{C'}(\mu',L')=J^ {\mu'}(C')=W_{\mu'}(C')$ and $g'\leqslant \min\{\mu', r\}$ (see Lemma \ref {lem:class2}, (i)). Then  $Z$ is $f^*(J(C'))$--stable,  and, since $a$ is the maximal dimension of an abelian subvariety of $J(C)$ for which $Z$ is stable
  and  $a=\dim (\bar A)\leqslant g'$, it follows   $a=g'$. In conclusion $\mu'=a=r=g'$, and we are in case (ii). 
 
 Assume now  $\mu>0$. 
 The general  point  of $Z$ (which is a component $\Gamma(\underline m,L)$)  is smooth for $\Theta$, hence it  corresponds to a linearly isolated, effective divisor $D$ of degree $g-1$, which is reduced (see Lemma \ref {lem:irreducible}, (ii)) and  can be written in a unique way  as $D=M+f^*(N)$, where $M$ and $N$ are effective divisors, with $\deg(M)=\mu$, $\deg (N)=\mu'$ and $M':=f_*(M)$   reduced.  
  So there is a rational map $h\colon Z \dasharrow J^{\mu}(C)$ defined by $D=M+f^*(N)\mapsto M$.
  
  Assume $\mu\le r$. By Lemma \ref{lem:irreducible} the image of $h$ is $W_{\mu}(C)$. Identify $A$ with its general translate inside $Z$. Then we have a morphism  $h|_A\colon  A\to J(C)$ whose image we denote by $\bar A$. Then  $W_{\mu}(C)$ is $\bar A$--stable. 
 Since  $W_{\mu}(C)$ is birational to $C(\mu)$, which is of general type because  $\mu<g$, then $\bar A=\{0\}$. It follows that  each component of the general fibre of $h$ is $A$--stable, in particular $r-\mu\geqslant a\geqslant 1$.
  
Take $M\in W_{\mu}(C)$ general and set $L'':=L'(-M')$. Since $L'$ is birational and $M'\in W_\mu(C')$ is general, with $\mu< r=\dim (L')$, then $L''$ has dimension $r-\mu\geqslant 1$ and it is base point free. It is also birational as soon as $r-\mu\geqslant 2$. 

Assume first $r-\mu\geqslant 2$. Then $C'(\mu',L'')$ is irreducible by Lemma \ref {lem:irreducible}, (iv), and  there is a  birational morphism  $C'(\mu',L'')\to h^ {-1}(M)\subset J^{g-1}(C)$  factoring through the map $C'(\mu',L'')\to J^ {\mu'}(C')\into J^{g-1}(C)$, where the last inclusion is  translation by $M$. Namely, up to a translation, $h^{-1}(M)=\Gamma(\mu',L'')$. In particular, $\dim (\Gamma(\mu', L''))=\dim (C(\mu', L''))=r-\mu$, hence $\mu'\ge r-\mu$.
  
Remember that $h\inv(M)=\Gamma(\mu', L'')$ is $A$--stable for $M\in W_{\mu}(C)$ general. By Corollary \ref {cor:gen} one has $\Gamma_{C'}(\mu',L'')=W_{\mu'}(C')=J^ {\mu'}(C')$, so $g'\leqslant \min\{\mu', r-\mu\}=r-\mu$. On the other hand, since $h^ 0(\mathcal O_C(D))=1$, one has also  $h^ 0(\mathcal O_C(f^ *(N)))=1$, hence $\mu'\leqslant g'$ and we conclude  
that  $\mu'=g'=r-\mu$.
 The same argument as above yields $a=g'$ and we are again in case (ii). 

If $r-\mu=1$, then $a=1$, $r=g-2$, $\mu=g-3$ and $\mu'=1$. On the other hand  $L$ is cut out on the canonical image of $C$ by the hyperplanes through the point $p_A$ which is the projectivized tangent space to $A$ at the origin. Then $\phi_L: C\to  \bar C$ is the projection from $p_A$, $\bar C$ it is a normal elliptic curve, and we are again  in case (ii).

Assume now $\mu>r$ and keep the above notation. In this case the map $h\colon  Z\dasharrow \bar Z:=h(Z)$ is generically finite, $\bar A$ is isogenous to $A$ and $\bar Z$ is  $\bar A$--stable. By  \cite[Lemma 3.1] {df}, we have $g-1\geqslant \mu\geqslant \dim (\bar Z)+\dim (\bar A)=r+a=g-1$, thus $\mu=g-1$, so we are in case (i).

Finally consider the case $\nu=3$.   Write the general $D\in \Gamma(\underline m, L)$ as $D=M_1+M_2+f^*(N)$,   where $M_1$ is reduced of degree $\mu$, $f_*(M_2)=2M'_2$ with $M'_2$ reduced of degree $\mu'$ and, as above, $\mu_3=\deg(N)\leqslant g'$. Set $\tau=\mu+\mu'$
 and consider the rational map $h\colon Z \dasharrow W_{\tau}(C)\subseteq J^{\tau}(C)$ defined by $D\mapsto M_1+f^*(M'_2)-M_2$.

If $\tau\leqslant r$,  arguing as above (and keeping a similar notation) one sees that $h(Z)=W_{\tau}(C)$, the general fibre of $h$ is $A$--stable, hence $r-\tau\geqslant a$. However $a=1$ and $\phi_L$ non--birational, forces, as we have seen, $\nu=2$, which is not the case here. Hence we have $a\geqslant 2$. We consider now $L''=L'(-f_*(M_1)-M'_2)$, which has dimension $r-\tau\geqslant 2$ and is base point free and birational, so  $C'(\mu_3,L'')$ is irreducible. 
The general fiber $h^ {-1}(h(D))$ of $h$ is isomorphic to $\Gamma_{C'}(\mu_3, L'')$ and is $A$--stable. In particular $\dim(\Gamma_{C'}(\mu_3, L''))=r-\tau$, hence $\mu_3\ge r-\tau$. By Corollary \ref {cor:gen}, we have $\Gamma_{C'}(\mu_3, L'')\cong W_{\mu_3}(C')\cong J^ {\mu_3}(C')$, so $g'\leqslant \min\{\mu_3,r-\tau\}=r-\tau$, hence $g'\le r-\tau \leqslant \mu_3\leqslant g'$, thus $\mu_3=r-\tau=g'$. In addition, 
as above, we have $a=g'$, so that $A=f^*(J(C'))$. 
Then $g-1-a=r= \tau+\mu_3=\mu+\mu'+\mu_3$. On the other hand $g-1=\mu+2\mu'+3\mu_3$. This yields $\mu'+2\mu_3= a= \mu_3$, hence $\mu'=\mu_3=0$, which is not possible. 

If $\tau>r$ then $h\colon  Z\dasharrow \bar Z:=h(Z)$ is generically finite, $\bar A$ is isogenous to $A$ and $\bar Z$ is  $\bar A$--stable. By  \cite[Lemma 3.1] {df}, we have $g-1\geqslant \tau\geqslant \dim (\bar Z)+\dim (\bar A)=r+a=g-1$, thus $\tau=g-1$, contradicting $\tau\leqslant  \deg (L')=\frac 2 3(g-1)$. \end{proof}

  \begin{lem}\label{lem:degL}  If $\nu\geqslant 4$ then $\nu=4$ and
 either\\
 \begin{inparaenum} [(i)]
 \item there is a degree 2 map $\psi\colon C\to E_1$ with $E_1$ a genus $r$ hyperelliptic curve such that $Z=A=\psi^*(J^ r(E_1))$; or\\
 \item there is a faithful $\Z_2^2$--action on $C$ with rational quotient;  denoting by  $f_i\colon C\to E_i$ (for $1\leqslant i\leqslant 3$) the quotient map for the three non--trivial involutions of $\Z_2^2$, with $E_i$ of genus $g_i$, and $g_1\geqslant g_2\geqslant g_3$, then $g_1=r+1$, $g_2+g_3=r$ and  $Z=A=f_2^*(J^ {g_2}(E_2))\times f_3^*(J^ {g_3}(E_3))$ is the Prym variety associated to $f_1$.
 \end{inparaenum}
 
 \end{lem}
\begin{proof} 
Since $L$ is base point free by Lemma \ref{lem:free} and it is not birational by Lemma \ref{lem:nobir}, we have $\delta=\frac {2g-2}\nu\geqslant r\geqslant  \frac{g-1}{2}$, hence   $\nu\leqslant 4$.

Assume $\nu=4$. Then  $\bar C$ is a   curve of degree $\frac {g-1}{2}$ spanning a projective space of dimension $r\geqslant \frac{g-1}{2}$. Hence $r=a=\frac{g-1}{2}$, $Z=A$,  and $\phi_L=f$ is the composition of a $g^1_4$ (that we denote by $\mathcal L$) with the degree $r$ Veronese embedding $\pp^1\to \pp^r$.

Let $\underline m=(1^ {\mu_1},\ldots, 4^ {\mu_4})$ be the partition of $2r$ such that $Z\subseteq \Gamma(\underline m,L)$. 

\begin{claim}
One has $\underline m=(2^ r)$. 
\end{claim}

\begin {proof} [Proof of the Claim] Assume by contradiction this is not the case, so that one among $\mu_1,\mu_3,\mu_4$ is non--zero.   
We have  $r=\dim(\Gamma(\underline m,L))=\mu_1+\cdots+\mu_4$, 
because $\bar C$ is a rational normal curve of degree $r$ in $\pp^r$,   and $2r=g-1=\mu_1+2\mu_2+3\mu_3+4\mu_4$, hence
$\mu_1=\mu_3+2\mu_4$ and $\mu_1>0$. 
So we may write the general divisor $D\in Z$ as $D=M+N$, where $M':=f_*(M)$ is reduced of degree $\mu:=\mu_1$. 
Then we proceed as  in the proof of Lemma \ref{lem:key}.

Consider the rational map $h\colon  Z=A\dasharrow J^{\mu}(C)$ defined by $D=M+N\mapsto M$.  It extends to a morphism and $\bar A:=h(A)$ is an abelian variety contained in $W_{\mu}(C)$.  
Since $\mu\leqslant r=\frac{g-1}{2}$, one has $\bar A=W_{\mu}(C)$, which is  impossible.
\end{proof}

Let $E:=C(2,L)$. Assume first $C$ is not hyperelliptic. Then there is a birational (dominant) morphism $E(r)\to \Gamma((2^ r), L)$ (see Lemma \ref {lem:tetra2}, (i)). 
If $E$ is irreducible, then  by Corollary \ref {cor:gen2} its image generates $J^2(C)$, hence also  $\Gamma((2^ r), L)=Z$ generates $J^ {2r}(C)$ and we obtain a contradiction.

So $E$ is reducible and we apply Lemma \ref{lem:tetra} and  \ref{lem:tetra2}.  Suppose we are in case (b) of Lemma \ref{lem:tetra} and (ii) of Lemma \ref{lem:tetra2}. Then  $A=Z=\Gamma((2^ r), L)$ is birational to $E_1(r_1)\times E_2(r_2)$, for  non--negative integers $r_1,r_2$. However, if $r_i>0$, then $r_i=g_i$, for $i\in \{1,2\}$.
Since $2g_2\geqslant g$, one has $r_2=0$ because $2r=g-1<g$. Hence 
$A$ is birational to $E_1(g_1)$ and we are in case (i).

Consider now case (c) of Lemma \ref{lem:tetra} and (iii) of Lemma \ref{lem:tetra2}.   We claim that  $J(C)$ is isogenous to $J(E_1)\times J(E_2)\times J(E_3)$.  Indeed, consider the representation of $\Z_2^2$ on $H^0(K_C)$. Since  $C/\Z_2^2$ is rational, we have $H^0(K_C)=V_{\chi_1}\oplus V_{\chi_2}\oplus V_{\chi_3}$, where for $i=1,2,3$ the non trivial character $\chi_i$  of $\Z_2^2$ is  orthogonal to the involution $\iota_i$ such that $C/\iota_i\cong E_i$, and $\Z_2^2$ acts on $V_{\chi_i}$ as  multiplication by $\chi_i$. Thus $V_{\chi_i}$ is the tangent space to $f_i^*(J(E_i))$. 

Recall  that there are three non--negative integers $r_1,r_2,r_3$ such that $A$ is birational to $E_1(r_1)\times E_2(r_2)\times E_3(r_3)$. If $r_i>0$, then $r_i=g_i$ (for $1\leqslant i\leqslant 3$). Since $g_1+g_2+g_3=g=2r+1$, at least one of the integers $r_i$ is zero. If two of them are zero,  we are again in case (i). So assume  $r_1=0$ and $r_2,r_3$  non--zero. Then $r=g_2+g_3$ and $g_1=r+1$. Moreover  the involution $\iota_1$ acts on $A$ as multiplication by $-1$, hence $A$ is the Prym variety of $f_1\colon C\to E_1$, and we are in case (ii). 

The case $C$ hyperelliptic can be treated similarly. In case (b) of Lemma \ref{lem:tetra} and (ii) of Lemma \ref{lem:tetra2}, $E_1$ is rational (hence it is contracted to a point  by $j$), and 
$A=Z=\Gamma((2^ r), L)$ is birational to $E_2(r)$. Then $r=g_2$ and we reach a contradiction since $2g_2\geqslant g$.  In case 
(c) of Lemma \ref{lem:tetra} and (iii) of Lemma \ref{lem:tetra2} one has $g_3=0$ and the above argument applies with no change.\end{proof}

\begin{proof}[Proof of Theorem \ref{thm:main}]
By Lemma \ref{lem:degL}, we may assume $\nu=2$.
Let  $\underline m$  be the partition of $2r$ such that $Z\subseteq \Gamma(\underline m, L)$. By Lemma \ref{lem:key}, it is enough to consider 
 the case $\nu=2$ and   $\underline m=(1^ {g-1})$, i.e., case (i) of that lemma. Then $Z$ is contained in the kernel $P$ of $f_*\colon J^{g-1}(C)\to J^{g-1}(C')$ (namely $P$ is the generalized Prym variety associated with $f$). 
  The space $H^0(K_C)$ decomposes under the involution $\iota$ associated with $f$ as  $H^0(K_{C'})\oplus V$, where $V$ is the space of antiinvariant $1$-forms. Hence $V^*$ is the tangent space to $P$,   the tangent space $T$ to $A$ is also contained in $V^*$ and the linear series $L$ is equal to $\pp(T^{\perp})\supseteq \pp(H^0(K_{C'}))$. On the other hand, by construction $\iota$ acts trivially on $L$, hence $T=V^*$ and thus $A=P$.  This implies that $Z=A=P$, $g=2r+1$ and $f$ is unramified with $g'=r+1$.  \end{proof}
  
  \subsection{General Brill--Noether loci}\label{ssec:GBN} The proof of the general formula \eqref {eq:DF} in \cite{df} uses an argument which is  useful to  briefly recall.
 Let $Z\subseteq W^ s_d(C)\subsetneq J^ d(C)$ be $A$--stable and $L\in Z$ general, so $L$ is a special  $g_d^s$, which we may assume to be complete, so $d\geqslant 2s$. Let $F_L$ be the fixed divisor of $L$ and, if $s>0$, 
 let $L'$ be the base point free residual linear series. Set $d'=\deg(L')$. Since $L'$ is also special, we have $d'\geqslant 2s$.  Then, by Lemma \ref {lem:tetra}, (iii), one has the birational map $\mathfrak r\colon  C(d'-s,L')\dasharrow C(s,L')\cong C(s)$. 

Consider the  morphism $j\colon  C(d'-s,L')\times C(g-1-d+s)\to J^{g-1}(C)$ such that $j(D,D')$ is the class of $D+F_L+D'$ (if $s=0$, we define $j\colon C(g-1-d)\to J^{g-1}(C)$ by $D'\mapsto F_L+D'$).
  If $(D,D')$ is general in $C(d'-s,L')\times C(g-1-d+s)$, the divisor $D+F_L+D'$ is linearly isolated, hence $j$  is generically finite onto its image $Z_L$, which  therefore has dimension $g-1-d+2s$. Consider the closure $Z'$ of  the union of the  $Z_L$'s, with $L\in Z$ general, which is $A$--stable. One has $Z'\subseteq \Theta$ and the above discussion yields $r':=\dim(Z')=r+g-1-d+2s$. Therefore $r'+a\leqslant g-1$ if and only if  \eqref {eq:DF} holds, with equality if and only if  equality holds in \eqref {eq:DF}.  
  
  \begin{cor}\label{for:ten} Let $C$ be a curve of genus $g$. Let $A\subsetneq J(C)$ be an abelian variety of dimension $a>0$ and $Z\subseteq W^ s_d(C)
  \subsetneq J^ d(C)$ an irreducible,  $A$--stable variety of dimension $r=d-2s-a$.
 Assume $(d,s)\neq (g-1,0)$. Then 
  there is a degree 2 morphism $\varphi\colon C\to C'$, with $C'$ a smooth curve of genus $a$ with $g>2a+1$, such that 
$A=\varphi^*(J(C'))$ and either\\
\begin{inparaenum}[(i)]
\item $Z=W_{d-2a-2s}(C)+\varphi^*(J^{a+s}(C'))$, or\\
\item $C$ is hyperelliptic and $Z=\varphi^*(J^{a}(C'))+W_{d-2s-2a}(C)+ W^ s_{2s}(C)$. 
\end{inparaenum}
\end{cor}

\begin{proof} We keep the same notation as above. 

We apply Theorem \ref {thm:main}  to $Z'$. Since
$d=r+a+2s\geqslant 2(a+s)$ and $(d,s)\neq (g-1,0)$, then $2a<g-1$, hence case (b) of Theorem \ref {thm:main} does not occur for $Z'$. 
So we have a degree 2 morphism $\varphi\colon C\to C'$, with $C'$ a smooth curve of genus $a$ with $g>2a+1$, such that 
$A=\varphi^*(J(C'))$ and $Z'=W_{g-1-2a}(C)+\varphi^*(J^a(C'))$. The case $s=0$ follows right away and we are in case (i). So we assume $s>0$ from now on. 

If $L'$ is composed with the involution $\iota$ determined by $\varphi$, we are again in case (i). So assume $L'$ is not composed with $\iota$. The general $D\in C(d'-s,L')$ contains no fibre of $\varphi$ and the same happens for the general $D'\in C(g-1-d+s)$. Since $D+F_L+D'$ corresponds to the general point of $Z$, and contains exactly $a$ general fibres of $\varphi$, then $F_L$ has to contain these $a$ fibres, whose union we denote by $F$. Moreover, the description of $Z$ implies that $D+D'+(F_L-F)$ is a general divisor of degree $g-1-2a$, in particular $D$ is a general divisor of degree $d'-s$. But $D$ is also general in $C(d'-s,L')$, and this implies $d'-s\leqslant s$. On the other hand $d'\geqslant 2s$, hence $d'=2s$. Then, either $L'$ is the canonical series of $C$ or $C$ is hyperelliptic and $L'$ is the $s$--multiple of the $g^ 1_2$. However the former case does not occur since by construction $d'=\deg (L')\le g-1$,
hence  we are in case (ii).   \end{proof}

     \end{document}